\documentclass[12pt,a4paper]{amsart}
\usepackage{graphics}
\usepackage{epsfig}
\usepackage{graphicx}
\theoremstyle{plain}
\usepackage{amssymb}
\usepackage{color}
\advance\hoffset-20mm \advance\textwidth40mm

\newtheorem{theorem}{Theorem}
\newtheorem{lemma}{Lemma}

\theoremstyle{definition}

\newtheorem*{definition*}{Definition}

\newtheorem{remark}{Remark}

\DeclareMathOperator{\rk}{rk}

\DeclareMathOperator{\ad}{ad}

\begin{document}
\sloppy
\title[A family of maximal subalgebras of the Lie algebra~$W_n(K)$]
{A family of maximal subalgebras of the Lie algebra~$W_n(K)$}
\author
{Y.Chapovskyi, A.Petravchuk}
\address{Institute of Mathematics, National Academy of Sciences of Ukraine,
Tereschenkivska street, 3, 01004 Kyiv, Ukraine}
\email{safemacc@gmail.com}
\address{ Faculty of Mechanics and Mathematics,
Taras Shevchenko National University of Kyiv, 64, Volodymyrska street, 01033  Kyiv, Ukraine}
\email{ petravchuk@knu.ua, apetrav@gmail.com}

\date{\today}

\begin{abstract}
Let $K$ be an algebraically closed field of characteristic zero and
${P_n=K[x_1,\ldots,x_n]}$ the polynomial ring.
Any $K$-derivation $D$ on $P_n$ is of the form
${
D=\sum_{i=1}^n f_i(x_1,\ldots,x_n)\frac{\partial}{\partial x_i}
},$
where $f_i\in P_n.$ All such derivations form the Lie algebra
$W_n(K)$ over the field $K$.
We prove that for $s=1,\ldots,n-1$ the subalgebra
$
m_s(K)=\left\{
\sum_{i=1}^s f_i\frac{\partial}{\partial x_i}
+\sum_{j=s+1}^n g_j\frac{\partial}{\partial x_j}
\mid f_i\in P_s,\ g_j\in P_n
\right\}
$
is a maximal subalgebra of~$W_n(K)$.
The ideal
$
I_s=\left\{\sum_{j=s+1}^n g_j\frac{\partial}{\partial x_j}\right\}
$
of $m_s(K)$ is isomorphic to the Lie algebra
$P_s\otimes \mathrm{Der}(K[x_{s+1},\ldots,x_n])$
and   $m_s(K)/I_s\simeq W_s(K)$. The Lie algebra $W_n(K)$ is also the free module over the ring $P_n.$ Therefore,  for  any set $S\subseteq W_n(K)$ the rank $rk(S)$ (over $P_n$) is defined. 
Some properties of maximal subalgebras of rank $n$ in $W_n(K)$
are pointed out.
\end{abstract}
\maketitle


\section{Introduction}

Let $K$ be an algebraically closed field of characteristic zero,
$P_n=K[x_1,\ldots,x_n]$ the polynomial ring in $n$ variables, and $R_n=K(x_1, \ldots , x_n)$ the field of rational functions. Denote by 
$W_n=W_n(K)$ the Lie algebra of all $K$-derivations on $P_n$.
Any element $D\in W_n$ can be uniquely written in the form
\[
D=f_1(x_1,\ldots,x_n)\frac{\partial}{\partial x_1}
+\cdots+
f_n(x_1,\ldots,x_n)\frac{\partial}{\partial x_n},
\qquad f_i\in P_n,
\]
so the Lie algebra $W_n$ is also the free module of rank $n$ over the ring $P_n$ (with the standard basis $\frac{\partial}{\partial x_1}, \ldots , \frac{\partial}{\partial x_n} $).
From the viewpoint of differential geometry, $W_n$ is the Lie algebra
of polynomial vector fields on $K^n$.
The structure of the Lie algebra $W_n$ has been  studied by many authors
but  many open problems remain.
For example, the structure of maximal subalgebras was described only
in case $n=1$, that is,  for the Lie algebra $W_1$ (see, \cite{Bell}), while for $n\ge2$
only  limited information is available.
Maximal subalgebras of finite codimension from $W_n$ were studied in
\cite{Rudak1}, \cite{Rudak2}, maximal locally nilpotent and solvable subalgebras were
studied in \cite{ESS}, see also \cite{MP} and \cite{Bavula}.

In this paper, we point out a family of maximal subalgebras
$m_s(K),  s=1, \ldots, n-1$ of the Lie algebra $W_n=W_n(K)$ and prove that they are pairwise
non-isomorphic for different values of  $s$ (Theorem~\ref{th_1}).
We also study maximal subalgebras $M$ of the Lie algebra $W_n$ of rank $n$ over $P_n$. If $M$ is not a submudule of the $P_n$-module $W_n$ and contains nonzero ideals of rank $<n,$ then $M$ contains a unique polynomial ideal of some rank $<n$ of the Lie algebra $M$ 
(Theorem~\ref{th:poly}).

Notation used in this paper is standard. 
We also use the standard grading on $W_n :$ 
 ${W_n = \bigoplus_{j \ge -1} W^{[j]}_n}$,
where 
$$W^{[j]}_n = \left \{
\sum_{i=1}^n f_i \frac{\partial}{\partial x_i} \in W_n \mid 
f_i \text{ are either zero or homogeneous polynomials of degree} \  j+1 \right \}.$$ 
If $D\in W^{[j]}_n,$  we say that $D$ is homogeneous of degree $\deg D=j. $
 For a $P_n$-module M we define its rank $\rk M$ as $\rk M = \dim_{R_n} R_n \otimes_{P_n} M$.
 The rank of a subset of a $P_n$-module is the rank of the corresponding generated submodule. A subalgebra $L$  of the Lie algebra $W_n$ will be called a polynomial subalgebra if $L$ is a submodule of the  $P_n$-module $W_n$  (the latter means that $P_nL=L$.) 
 
\section{A family of maximal subalgebras of $W_n(K)$.}
The following statements can be verified immediately (see, for example,  \cite{Now}).

\begin{lemma} \label{lm:elem}
	Let $D_1, D_2 \in W_n(K)$ and $f, g \in P_n = K[x_1, \dots, x_n.]$ Then: 
	\begin{enumerate}
		\item $[f D_1, g D_2] = f D_1(g) D_2 - g D_2(f) D_1 + fg[D_1, D_2]$.
		\item If $[D_1, D_2] = 0$, then $[f D_1, g D_2] = f D_1(g) D_2 - g D_2(f) D_1$.
		\item If $E_n = \sum_{j=1}^{n} x_j \frac{\partial}{\partial x_j}$ is the Euler derivation, and $f$ is a homogeneous polynomial of degree $m$, then $E_n(f) = m f$.
		\item If $D \in W_n(K)$ is a homogeneous derivation of degree $k$, then $[E_n, D] = k D$.
	\end{enumerate}
\end{lemma}

The next statement seems to be known, but since we could not find a precise reference, we provide a complete proof.  

\begin{lemma} \label{lm:grading}
	Let $L$ be a subalgebra of the Lie algebra $W_n(K)$. If $E_n \in L$, then $L$ is a graded subalgebra with respect to the standard grading on $W_n(K)$.
\end{lemma}

\begin{proof}
	Recall that we use the  standard grading on $W_n$, $W_n(K) = \bigoplus_{i \ge -1} W_n^{[i]},$ where the homogeneous component  $W_n^{[i]}$ consists of all derivations whose polynomial coefficients are homogeneous  of degree $i+1$ or zero.
	Suppose, on the contrary, that the statement of the lemma is false. Then there exists a derivation $D \in L$ such that 
	$$D = D_{l_1} + \dots + D_{l_k}, l_1<l_2\cdots <l_k$$
	 with $k \ge 2$, where $D_{l_i} \in W_n^{[l_i]}$ are homogeneous components of degree $l_i$ such that all $D_{l_i}\not \in L.$ 
	Choose such a derivation of  the minimal possible length $k$.
	Since $E_n \in L$, by Lemma \ref{lm:elem}(4) we have  
	$$[E_n, D] = l_1 D_{l_1} + \dots + l_k D_{l_k} \in L.$$
	Therefore, $[E_n, D] - l_k D \in L.$ 	But then we have 
	\[
	[E_n, D] - l_k D = (l_1 - l_k)D_{l_1} + \dots + (l_{k-1} - l_k)D_{l_{k-1}}.
	\]
	Note that $l_i \neq l_k$ for $i < k$. If the element $ [E_n, D] - l_k D $  is non-zero, it has length $k-1$, which contradicts the minimality of  $k.$ 
	Thus all homogeneous components of $D$  belong to $L,$ and hence $L$ is graded. 
\end{proof}
Note that there are only a few known examples of maximal subalgebras of the  Lie algebra $W_n$ for $n>1,$ in contrast to the case $W_1,$ where all such subalgebras are described.  The next statement provides a series of maximal subalgebras of $W_n$.
	\begin{theorem}\label{th_1}
	\quad
\begin{enumerate}
	\item
	Let $s$ be an integer with $1\le s\le n-1$.
	Then the set
	\[
	m_s(K)=\left\{
	\sum_{i=1}^s f_i\frac{\partial}{\partial x_i}
	+\sum_{j=s+1}^n g_j\frac{\partial}{\partial x_j}
	\mid f_i\in P_s,\ g_j\in P_n
	\right\}
	\]
	is a maximal subalgebra of the Lie algebra $W_n$.
	
	\item
	The set
	\[
	I_s=\left\{\sum_{j=s+1}^n g_j\frac{\partial}{\partial x_j}\mid g_j\in P_n\right\}
	\]
	is a polynomial ideal of the Lie algebra $m_s(K)$ and
	$I_s\simeq P_s\otimes \mathrm{Der}(K[x_{s+1},\ldots,x_n])$.
	Besides, the quotient algebra $m_s(K)/I_s$ is isomorphic to the Lie algebra $W_s$.
	
	\item $I_s$ is a unique non-trivial ideal of the Lie algebra  $m_s(K)$.
	
	\item
	The subalgebras $m_s(K)$ and $m_t(K)$ are non-isomorphic for $s\ne t$,
	$s,t=1,\ldots,n-1$.
\end{enumerate}\end{theorem}

\begin{proof}[Proof of (1)]
	One can immediately check that $m_s(K)$ is a subalgebra of $W_n(K)$. Let $S$ be a   subalgebra of $W_n(K)$ strictly containing the subalgebra  $m_s(K)$.
	The subalgebra $m_s(K)$ obviously contains the Euler derivation $E_n = \sum_{i=1}^n x_i \frac{\partial}{\partial x_i}$, so 
	by Lemma \ref{lm:grading} the subalgebra $S$ is graded, $S = \bigoplus_{i \ge -1} S_i$, where  $S_i=S\cap W_n^{[i]} $ and $W_n^{[i]}$ denotes the standard grading component of the Lie algebra $W_n(K)$.
	
	Repeating some parts of the proof of Lemma \ref{lm:grading} one  can show that the set $S \setminus m_s(K)$ contains nonzero  homogeneous derivations. Fix such a derivation and denote it by $D$.
	Write $D = \sum_{i=1}^n h_i \frac{\partial}{\partial x_i}, $ where  $h_i\in P_n, i=1, \ldots , n.$ Since $D \notin m_s(K)$, there exists an index $j_0 \in \{1, \dots, s\}$ such that the coefficient $h_{j_0}$ depends on a variable $x_{i_0}, i_0 \ge s + 1,$ that is, $\frac{\partial h_{j_0}}{\partial x_{i_0}}\not =0$. 	
	Write  the polynomial $h_{j_0}$ as a polynomial in  $x_{i_0}$:
	\[
	h_{j_0} = u_0 + u_1 x_{i_0} + \dots + u_d x_{i_0}^d
	\]
	for some $d \ge 1$, where $u_t$ does not depend on $x_{i_0}$ for $t=0,\dots d$  and $u_d \neq 0.$
	
	Consider the  following homogeneous derivation:
	\[
	D_1 := 
	\left[ \dots \left[ D, \frac{\partial}{\partial x_{i_0}} \right], \dots, \frac{\partial}{\partial x_{i_0}} \right] = \sum_{l=1}^n q_l \frac{\partial}{\partial x_l}, 
	\]
	where $\partial/\partial x_{i_0}$  is applied  $d-1$ times).
	Note that the coefficient $ q_{j_0} $  can be written in the form $q_{j_0} = v_0 + v_1 x_{i_0}$ with $v_1 \neq 0$, and the polynomials $v_0, v_1$ do not depend on $x_{i_0}$.
	Clearly, $D_1 \in S \setminus m_s(K)$ as well.
	Without loss of generality, we may assume that $v_1 \in K^*$. Indeed, if $v_1 \notin K^*$,  consider
	a  monomial $x_1^{l_1} \dots x_{i_0}^0 \dots x_n^{l_n}$ of maximal total degree  in the polynomial $v_1$. Acting by the linear operator 
	$$\ad(\frac{\partial}{\partial x_1})^{l_1} \dots \ \ad(\frac{\partial}{\partial x_n})^{l_n}, \  l_i\geq 0, \  i=1, \ldots , n$$ on $D_1 $
	we can obtain a derivation satisfying the condition $v_1 \in K^*$.
	Since $D_1$ is  homogeneous, it must be a linear homogeneous derivation. Acting by the operator $\ad(x_{i_0} \frac{\partial}{\partial x_{i_0}})$
	on the derivation $D_1$ we obtain the derivation 
	\[
	D_2:= \left[x_{i_0} \frac{\partial}{\partial x_{i_0}}, D_1 \right] = \sum_{i=1}^n c_i x_{i_0} \frac{\partial}{\partial x_{i}}
	=  x_{i_0}\sum_{i=1}^n c_i \frac{\partial}{\partial x_{i}},
	\]
	where $c_i \in K$, $c_{j_0} \ne 0$.
	Again, $D_2 \in S \setminus m_s(K)$. We may assume that $c_i = 0$ for all $i > s$, since 
	$\sum_{i=s+1}^n c_i x_{i_0} \frac{\partial}{\partial x_{i}} \in m_s(K)$.
	Note that
	\[
	\left[x_{j_0} \frac{\partial}{\partial x_{i_0}}, D_2 \right] = 
	\left[ x_{j_0} \frac{\partial}{\partial x_{i_0}}, x_{i_0}\sum_{i=1}^{s} c_i \frac{\partial}{\partial x_i} \right] = 
	x_{j_0} \sum_{i=1}^{s} c_i \frac{\partial}{\partial x_i} - x_{i_0} c_{j_0} \frac{\partial}{\partial x_{j_0}}.
	\]
	Since the first term in the latter sum belongs to  the subalgebra $m_s(K)$ and $c_{j_0} \ne 0,$ we  conclude that 
	$x_{i_0} \frac{\partial}{\partial x_{j_0}} \in S$.
	
	Furthermore, for any $p, 1 \le p \le s $,  we have  the equality
	$$\left[ x_{j_0} \frac{\partial}{\partial x_p}, 
	x_{i_0} \frac{\partial}{\partial x_{j_0}} \right] = 
	-x_{i_0} \frac{\partial}{\partial x_p}.$$ 
	 
	So, we can find in the subalgebra $S$ all  derivations of the form $x_{i_0} \frac{\partial}{\partial x_p}$, $1 \le p \le s$.
	
	Take an arbitrary  element of the form $f(x_1, \dots, x_n) \frac{\partial}{\partial x_{i_0}}$ and let $j \le s$.
	Then 
	$$\left[ f(x_1, \dots, x_n) \frac{\partial}{\partial x_{i_0}}, 
	x_{i_0} \frac{\partial}{\partial x_{j}} \right] = 
	f(x_1, \dots, x_n) \frac{\partial}{\partial x_{j}} - 
	x_{i_0} \frac{\partial f}{\partial x_j} \frac{\partial}{\partial x_{i_0}}.$$
	Since the second term lies in $m_s(K)$, 
	we conclude that $f \frac{\partial}{\partial x_{j}} \in S$. Therefore, $S = W_n(K)$, which shows that  
	$m_s(K)$ is a maximal subalgebra of $W_n(K)$.
\end{proof}

\begin{proof}[Proof of (2)]
	It is easy to check that $I_s$ is an ideal of the Lie algebra $m_s(K). $  Let $D \in {I}_{s}$. Then  $D$ can be written in the form
	\[
	D = f_{s+1} \frac{\partial}{\partial x_{s+1}} + \dots + f_n \frac{\partial}{\partial x_n},
	\]
	where  $f_i\in P_n, i=s+1, \ldots , n.$ Note that any polynomial in $P_n = K[x_1, \dots, x_n]$ can be regarded  as a polynomial in variables $x_{s+1}, \dots, x_n$ with coefficients in $P_{s} = K[x_1, \dots, x_{s}]$. In particular, 
	\[
	P_n \cong P_{s} \otimes K[x_{s+1}, \dots, x_n].
	\]
	Therefore:
	\[
	{I}_{s} \cong P_{s} \otimes \left( K[x_{s+1}, \dots, x_n] \frac{\partial}{\partial x_{s+1}} + \dots + K[x_{s+1}, \dots, x_n] \frac{\partial}{\partial x_n} \right) \cong P_{s} \otimes W_{n-s}.
	\]
	The quotient algebra $m_s(K) / {I}_{s}$ is clearly isomorphic to the Lie algebra generated by the first $s$ partial derivatives with coefficients in $P_{s},$ since the terms involving $\partial/\partial x_j$ for $j > s$ lie in the ideal ${I}_{s}$.
	Hence, $m_s(K) / {I}_{s}\simeq  W_{s}$.
	\end{proof}
\begin{proof}[Proof of (3)]
	Proceed step by step.
	\begin{enumerate}
		\item $I_s$ is a maximal ideal of the Lie algebra $m_s(K) $.
		Indeed, $m_s(K) / I_s \simeq W_s$, which is a simple Lie algebra.
		
		\item $I_s$ is a minimal ideal of $ m_s(K).$ To show this 
take an arbitrary nonzero element $D \in I_{s}$. Consider the ideal $J$ in $m_{s}(K)$ generated by $D$. Let us  show that ${J} = I_{s}$. Indeed, applying the  adjoint operators $\ad \frac{\partial}{\partial x_{i}}, i=1, \ldots , n$ on $D$ sufficiently many times we can obtain the element
\begin{equation*}
	\overline{D} = c_{s+1} \frac{\partial}{\partial x_{s+1}} + \dots + c_{n} \frac{\partial}{\partial x_{n}}, \quad c_{i} \in K.
\end{equation*}
The elements  $\frac{\partial}{\partial x_{1}}, \dots, \frac{\partial}{\partial x_{n}}$ lie in $ m_{s}(K)$, so  $\overline{D} \in {J}$.  Note that $\overline{D} $ can be chosen nonzero since $ D\not =0,$  so for some $k, s+1\leq k\leq n $ the coefficient $c_k$ is nonzero. It is obvious that $c_k^{-1}x_k\frac{\partial}{\partial x_k}\in m_s(K). $ Therefore $[D, \  c_k^{-1}x_k\frac{\partial}{\partial x_k}]=\frac{\partial}{\partial x_k}\in J.$
 Then for an arbitrary derivation of the form $f(x_1, \ldots , x_n)\frac{\partial}{\partial x_{i}}, s+1\leq i\leq n$  one can write the derivation $-\int f(x) dx_{k} \frac{\partial}{\partial x_{i}}$  and this derivation lies in  the ideal $I_s.$ Then we obviously have
\begin{equation*}
	f(x) \frac{\partial}{\partial x_{i}} = \left[ -\int f(x) dx_{k} \frac{\partial}{\partial x_{i}}, \frac{\partial}{\partial x_{k}} \right] \in {J}.
\end{equation*}
This demonstrates that ${J} = I_{s}$.
	\item Uniqueness.
	If we assume there exists a maximal ideal ${J} \neq I_{s}$, then the sum ${J} + I_{s}$ must be $m_{s}(K)$. Therefore, we have
\begin{equation*}
	W_{s}(K) \simeq m_{s}(K) / I_{s} = ({J} + I_{s}) / I_{s} \simeq {J} / ({J} \cap I_{s})
\end{equation*}
Since $I_{s}$ is a minimal ideal of the Lie algebra $m_s(K)$ it holds  ${J} \cap I_{s} = 0.$ The latter   implies $m_{s}(K) \simeq W_{s}(K) \oplus I_{s}$. 

\noindent
In~\cite{MP}, it was proved that the derived length of solvable subalgebras in $W_n(K)$ is at most  $2n$. The known example of solvable subalgebras that achieves this bound was pointed out in~\cite{Martello}, this is the solvable subalgebra
$$s_n(K)=(P_0+x_1P_0)\frac{\partial}{\partial x_1}+\dots+(P_{n-1}+x_nP_{n-1})\frac{\partial}{\partial x_n},$$ 
see also \cite{ESS}.
Denote by $d^{\star}(L)$ the maximum of derived lengths of solvable subalgebras  from a Lie algebra $L$.  Note that the Lie algebra $m_s(K)$ contains the subalgebra $s_n(K),$ so $$d^{\star}(m_s(K))=d^{\star}(W_n(K))=2n.$$
But we also have
\begin{equation*}
	d^{\star}(m_s(K))=d^{\star}(W_{s}(K) \oplus I_{s}) = \max(d^{\star}(W_{s}(K)), d^{\star}(I_{s})) = \max(2s, 2n - 2s) \neq 2n
\end{equation*}
The obtained contradiction proves the uniqueness of the ideal $I_s$ in the Lie algebra $m_s(K).$ 
	\end{enumerate}
\end{proof}
\begin{proof}[Proof of (4)]
	 Suppose, on the contrary, that  $m_s(K)\simeq m_t(K)$ for $s\ne t$ and let 
	$\varphi:m_s(K)\to m_t(K)$ be an isomorphism.
	By the above proven  $I_s$ is a maximal ideal of the Lie algebra $m_s(K)$.
	Therefore, $\varphi(I_s)=I_t$ and hence $I_s\simeq I_t$.
	By the above mentioned
	\[
	I_s\simeq W_{n-s}\otimes P_s \quad\text{and}\quad
	I_t\simeq W_{n-t}\otimes P_t.
	\]
	It is easy to see that  $$d^{\star}(W_{n-s}\otimes P_s)=d^{\star}(W_{n-s})=2(n-s)$$ and $$d^{\star}(W_{n-t}\otimes P_t)=d^{\star}(W_{n-t})=2(n-t).$$
	Therefore  $I_s\not\simeq I_t$, which implies  $m_s(K)\not\simeq m_t(K)$ for $s\ne t$.
	The proof is complete.
\end{proof}

\begin{remark}
	One may also consider a partition of the set $\{1,\ldots,n\}$ into a disjoint union
	$$\{1,\ldots,n\}=I\cup J, I=\{i_1,\ldots,i_s\}, J=\{i_{s+1},\ldots,i_n\}$$
	and define the  subalgebra
	\[
	m_I(K)=K[x_{i_1},\ldots,x_{i_s}]\frac{\partial}{\partial x_{i_1}}
	+\cdots+
	K[x_{i_1},\ldots,x_{i_s}]\frac{\partial}{\partial x_{i_s}}
	+P_n\frac{\partial}{\partial x_{i_{s+1}}}
	+\cdots+
	P_n\frac{\partial}{\partial x_{i_n}}.
	\]
	This subalgebra is also maximal in $W_n$ since it is conjugate to
	$m_s(K)$ via the automorphism $\theta\in \mathrm{Aut}\,P_n$ defined by
	$\theta (x_j)=x_{i_j}$, $j=1, \ldots , n$.
\end{remark}

\begin{remark}
	One could  attempt to generalize the previous theorem: for example, 
	for integers $s_1, s_2, 1\leq s_1<s_2<n$ one can consider the $K$-subspace of the Lie algebra $W_n$ of the form
	$$m_{s_1, s_2}:= P_{s_1}\frac{\partial}{\partial x_{1}}+\cdots +P_{s_1}\frac{\partial}{\partial x_{s_1}}+P_{s_2}\frac{\partial}{\partial x_{s_1+1}}+\cdots +P_{s_2}\frac{\partial}{\partial x_{s_2}}+  $$
	$$+P_{n}\frac{\partial}{\partial x_{s_2+1}}+\cdots +P_{n}\frac{\partial}{\partial x_{n}}.  $$
	 This subspace is a Lie subalgebra of $W_n$, but it is not maximal: $ m_{s_1, s_2}$ is strictly contained in the subalgebra $m_{s_1}$ defined earlier.
	
\end{remark}

\section{Polynomial maximal subalgebras of $W_n(K)$}

To study maximal subalgebras of the Lie algebra $W_n(K),$ it is useful to consider some general properties of polynomial subalgebras of
this Lie algebra.
There is a natural way to construct such subalgebras.
Let $I$ be a nonzero ideal of the ring $P_n$ and let $L$ be an arbitrary nonzero subalgebra of $W_n$.
Define the next vector space over the field $K$ consisting of finite sums of products 
\[
IL=\left\{\sum h_i D_i\mid h_i\in I,\ D_i\in L\right\} .
\]
Then $IL$ is a subalgebra of $W_n$ and, if $I=P_n$ we clearly have   $L\subseteq P_nL.$ 
It is obvious that $IL$ is a polynomial subalgebra  and
$\mathrm{rk}\,IL=\mathrm{rk}\,L$.
Since $W_n$ is a Noetherian module over $P_n$, every polynomial subalgebra
$L$ of $W_n$ is finitely generated as a $P_n$-module, that is,  there exist
elements $D_1,\ldots,D_m$ such that $ L=P_n D_1+\cdots+P_n D_m.$ We consider  maximal subalgebras of rank $n$ from the Lie algebra $W_n(K).$ 

\begin{theorem} \label{th:poly}
	Let $M$ be a maximal subalgebra of the Lie algebra $W_n$ with
	$\mathrm{rk}\,M=n$.
	Then
	\begin{enumerate}
		\item
		if $M$ is not a polynomial subalgebra and contains a nonzero ideal of
		rank  $\leq n-1$, then $M$ has a unique maximal (by inclusion) polynomial
		ideal $M_0$ with $rk(M_0)=k$ for some $k<n$ such that $M/M_0$ does not contain any ideals of rank
		$<n-k$.
		\item
		if $M$ is a polynomial subalgebra then $M$ contains a subalgebra
		$I W_n$ of rank $n$ for some ideal $I\in P_n$;
		
	\end{enumerate}
\end{theorem}
\begin{proof}
	(1) Let the maximal subalgebra $M$ be non-polynomial, and let $J\subset M$ be a nonzero ideal of $M$
	of rank $<n$.
	Then the  product $P_nJ$ is a polynomial subalgebra of $W_n$, and
	$[M,P_nJ]\subset P_nJ$.
	Therefore, $M_1=M+P_nJ$ is a subalgebra of $W_n$.
	Note that $M_1\ne W_n$.
	Indeed, if $M_1=W_n$, then the nonzero subalgebra $P_nJ$ would be  a proper ideal
	of the Lie algebra $W_n$.
	This is impossible, since  $W_n$ is a simple Lie algebra.
	Thus $M_1=M$,  and  hence $P_nJ\subset M$.
	We have therefore  shown that every ideal $J$ of $M$ of rank $<n$ is contained
	in the polynomial ideal $P_nJ$ of the Lie algebra $M$.
	Denote by $M_0$ the maximal polynomial ideal of $M$. Such an ideal exists, since  the sum of any family of polynomial ideals of $M$
	is again a polynomial ideal of $M$.
	Then $M_0$ is the unique maximal polynomial ideal of $M$, and it
	contains all ideals of rank $<n$ of the Lie algebra $M$. Denote $k=rk(M_0), $ where the rank is understood  as  the rank of a submodules of the $P_n$-module 	$W_n$.
	Finally, let $J/M_0$ be a nonzero ideal of the Lie algebra $M/M_0$
	with $\mathrm{rk}\,J/M_0<n-k.$ 
	Then $J$ is an ideal of $M$ of rank less than $n$, and $J\not \subseteq  M_0$ which
	contradicts the above results.
	Therefore, the  factor-algebra
	$M/M_0$ has no nonzero ideals of rank $<n-k.$ 
	
	(2) Now let $M$ be a maximal polynomial subalgebra of $W_n,$ and let $I=\{f\in P_n\mid fW_n\subset M\}$.
	It is easy to see that $I$ is an ideal of the polynomial ring $P_n$.
	We show that $I\ne0$.
Since $\mathrm{rk}\,M=n$, there exist elements
	$D_1,\ldots,D_n\in M$ that are linearly independent over the ring $P_n$.
	Therefore, the elements
	$\frac{\partial}{\partial x_i},D_1,\ldots,D_n, i=1, \ldots , n$ are linearly dependent over $P_n.$, Hence there exist
	polynomials $h_{i},h_{1i},\ldots,h_{ni}$, $h_i\ne0$, such that
	\[
	h_i\frac{\partial}{\partial x_i} = h_{1i}D_1+\cdots+h_{ni}D_n, \ i=1, \ldots , n.
	\]
	Set $h=h_1\cdots h_n$.
	Then $h\ne0$ and $hW_n\subset M$.
	This implies that  $h\in I$, and  hence $I\ne0$.
	The product $IW_n$ is clearly a polynomial subalgebra of rank $n$
	contained in  the subalgebra $M$.
		The proof is complete.
\end{proof}
\section*{Acknowledgments}
The first author of this work was partially supported by a grant
from the Simons Foundation (SFI-PD-Ukraine-00014586, Y.Y.C.).
Yevhenii Chapovskyi expresses his gratitude for the financial support of the National Academy of Sciences of Ukraine 
within the framework of the project 0125U002856 for young scientists.


%
\end{document}